\documentclass[12pt]{amsart}
\usepackage{mathrsfs,amsmath,amssymb,amsfonts,amsthm,latexsym}
\theoremstyle{plain}
\newtheorem{theorem}{Theorem}[section]
\newtheorem{proposition}[theorem]{Proposition}
\newtheorem{corollary}[theorem]{Corollary}
\newtheorem{lemma}[theorem]{Lemma}

\theoremstyle{definition}
\newtheorem{remark}[theorem]{Remark}

\title[Comparison Theorems]{Extensions of Perron-Frobenius Theory}
\author[N.~Gao]{Niushan Gao}
\address{Department of Mathematical and Statistical Sciences, University of Alberta, Edmonton, AB, Canada T6G\,2G1}
\email{niushan@ualberta.ca}
\date{\today}

\begin{document}

\begin{abstract}
The classical Perron-Frobenius theory asserts that for two matrices $A$ and $B$, if $0\leq B \leq A$ and $r(A)=r(B)$ with $A$ being irreducible, then $A=B$. This was recently extended in \cite{crad} to positive operators on 
$L_p(\mu)$ with either $A$ or $B$ being irreducible and
power compact. In this paper, we extend the results to irreducible operators on arbitrary Banach lattices.
\end{abstract}
\maketitle

\section{Introduction and Preliminaries}\label{sec1}
Throughout this paper, $X$ always denotes a real Banach lattice with $\dim X>1$, and $T$ always stands for a non-zero positive operator on $X$. Recall that a positive operator is said to be \textbf{ideal irreducible} if 
it has no non-zero proper invaraint closed ideals, and \textbf{band irreducible} if it has no non-zero proper invariant bands. These definitions of irreducibity coincide on order continuous Banach lattices; 
in particular, on $L_p$ spaces for $1\leq p<\infty$ and on $\mathbb{R}^n$. Moreover, for a positive matrix $A$, it is easily seen that $A$ is irreducible if and only if $A$ does not have a block form 
$\left[\begin{smallmatrix}
 A_{11}&A_{12}\\
0&A_{22}
 \end{smallmatrix}\right]
$ under any permutation of the standard 
basis.\par
The classical Perron-Frobenius theory (see Chapter~8, \cite{abr}) asserts the following.
\begin{theorem}\label{ree}
 Let $A>0$ be an irreducible matrix. Then its spectral radius $r(A)>0$, $r(A)$ is a simple root of its characteristic polynomial $f_A$, and $\ker(r(A)-A)=\mathrm{Span}\{x_0\}$ for some strictly positive vector $x_0$.
\end{theorem}

\begin{theorem}\label{cpft}Suppose $A$ and $B$ are two matrices such that $0\leq B\leq A$ and $r(A)=r(B)$. If $A$ is irreducible, then $A=B$.
\end{theorem}

Theorem~\ref{ree} has been generalized to positive operators on Banach lattices by various authors; see, e.g.~\cite{sawa,scha2,scha3,pag,sab,abr,grob1,grob2,kit}, etc. 
Extensions of Theorem~\ref{cpft} have also been considered by some authors; see, e.g., \cite{marek1,alek}. 
Results of this type are often referred to as \emph{comparison theorems}. 
The following comparison theorem is of interest.
\begin{theorem}[\cite{crad}]\label{crad0}Suppose $0\leq S\leq T$ on $L^p(\mu)$ where $1<p<\infty$ and $\mu$ is $\sigma$-finite. Suppose also $r(T)=r(S)$. Then $T=S$ if either $T$ or $S$ is irreducible and power compact.
\end{theorem}

In Section~\ref{sec3} of this paper, we generalize some known facts about positive eigenvectors of irreducible operators and their adjoint operators.
In Sections~\ref{sec5} and \ref{lastsss}, we provide several extensions of Theorem~\ref{cpft} to positive operators on arbitrary Banach lattices. In particular, we show in Corollary~\ref{ckm11} that Theorem~\ref{crad0} remains valid for general Banach lattices.
Moreover, we prove in Theorem~\ref{coo} that power compactness condition may be replaced with the (weaker) condition that the spectral radius is a pole of the resolvent.
\medskip

We write $\sigma(T)$ for the spectrum of $T$, $r(T)$ for the spectral radius of $T$, and $R(\cdot,T)$ for the resolvent of $T$. 
For $x\in X$, we write $x_+$ and $x_-$ for the positive and negative parts of $x$, respectively.
Recall that a positive operator is called \textbf{strongly expanding} (respectively, \textbf{expanding}) if it sends non-zero positive vectors to quasi-interior points (respectively, weak (order) units). 
A positive operator on $X$ is said to be \textbf{strictly} \textbf{positive} if it does not vanish on any non-zero positive vectors. It can be easily verified that if $T>0$ is $\sigma$-order continuous, then so is $\sum_1^\infty\frac{T^n}{\lambda^{n}}$ for each $\lambda>r(T)$. 
For background and notations on Banach spaces and Banach lattices, we refer the reader to \cite{conw,abr,ali,scha3}.\par

We will use the following well known properties of irreducible operators.

\begin{lemma}\label{chi1}Fix $\lambda>r(T)$. The following statements are equivalent:
\begin{enumerate}
 \item\label{chili1}$T$ is ideal irreducible;
\item\label{chili2}$\sum_1^\infty \frac{T^n}{\lambda^n}$ is strongly expanding;
\item\label{chili3}$\sum_1^\infty \frac{T^{n*}}{\lambda^n}x^*$ is strictly positive for any $x^*>0$.
\end{enumerate}
\end{lemma}

\begin{proof}The equivalence of \eqref{chili1} and \eqref{chili2} is shown in \cite{scha4}, p.~317. For $\eqref{chili2}\Leftrightarrow\eqref{chili3}$, 
simply note that $x^*\big(\sum_1^\infty \frac{T^n}{\lambda^n}x\big)=\big(\sum_1^\infty \frac{T^{n*}}{\lambda^n}x^*\big)(x)$, and that $y>0$ is a quasi-interior point if and only if $x^*(y)>0$ for any $x^*>0$.\end{proof}

\begin{lemma}\label{faci1}Let $T$ be ideal irreducible.
\begin{enumerate}
 \item\label{faci1i1} If $Tx=\lambda x$ for some $x>0$ and $\lambda\in\mathbb{R}$ then $x$ is a quasi-interior point and $\lambda>0$;\par
 \item\label{faci1i2} If $T^*x^*=\lambda x^*$ for some $x^*>0$ and $\lambda\in\mathbb{R}$ then $x^*$ is strictly positive and $\lambda>0$;\par
 \item\label{faci1i3} $T$ is strictly positive.
\end{enumerate}
\end{lemma}
\begin{proof}\eqref{faci1i1} It is clear that $\lambda\geq 0$. Now pick any $\delta>r(T)$. Then by Lemma~\ref{chi1}\,\eqref{chili2}, $\sum_1^\infty \frac{T^n}{\delta^n}x=\left(\sum_1^\infty\frac{\lambda^n}{\delta^n}\right)x$ 
is a quasi-interior point. Hence $\lambda>0$ and $x$ is a quasi-interior point. \eqref{faci1i2} can be proved similarly by using Lemma~\ref{chi1}\,\eqref{chili3}. \eqref{faci1i3} follows immediately from \eqref{faci1i1}.
\end{proof}

\begin{lemma}\label{chb2}Let $T>0$ be $\sigma$-order continuous. Fix $\lambda>r(T)$. The following two statements are equivalent:
\begin{enumerate}
 \item\label{chb2i1}  $T$ is band irreducible;
\item\label{chb2i2}  $\sum_1^\infty \frac{T^n}{\lambda^n}$ is expanding.\\
Any of these two implies the following:
\item\label{chb2i3} $\sum_1^\infty \frac{T^{n*}}{\lambda^n}x^*$ is strictly positive for any $\sigma$-order continuous $x^*>0$.
\end{enumerate}
\end{lemma}
\begin{proof}$\eqref{chb2i1}\Leftrightarrow\eqref{chb2i2}$ can be proved similarly as for ideal irreducible operators; see \cite{scha4}, p.~317. For the last assertion, simply note that if $y>0$ is a weak unit, 
then $x^*(y)>0$ for any $\sigma$-order continuous $x^*>0$.\end{proof}
\begin{lemma}\label{facb2}Let $T$ be band irreducible and $\sigma$-order continuous.
\begin{enumerate}
\item\label{facb2i1} If $Tx=\lambda x$ for some $x>0$ and $\lambda\in\mathbb{R}$, then $x$ is a weak unit and $\lambda>0$;
\item\label{facb2i2} If $T^*x^*=\lambda x^*$ for some $\sigma$-order continuous $x^*>0$ and $\lambda\in\mathbb{R}$, then $x^*$ is strictly positive and $\lambda>0$;
\item\label{facb2i3} $T$ is strictly positive.
\end{enumerate}
\end{lemma}
The following generalization of Krein-Rutman theorem will be used; see Proposition~4, \cite{scha2}.
\begin{lemma}\label{seig}If $r(T)$ is a pole of $R(\cdot,T)$, then the leading coefficient of the Laurent expansion of $R(\cdot,T)$ at $r(T)$ is positive, and $T$ as well as $T^*$ has a positive eigenvector for $r(T)$.
\end{lemma}

\section{Positive Eigenvectors of Irreducible Operators}\label{sec3}
In this section, we establish some basic lemmas. We will look at positive eigenvectors of irreducible operators. In particular, we are interested in positive eigenfunctionals of their adjoint operators.\par
Recall that we always assume $T>0$.
The following theorem is well known and can be found in Theorem~5.2, p.~329, \cite{scha3} for ideal irreducible operators and in Theorem 4.12, \cite{grob2} for band irreducible operators.
\begin{theorem}\label{eig02}Suppose $Tx_0=r(T) x_0$ and $T^*x_0^*=r(T) x_0^*$ for some $x_0>0$ and some $x_0^*>0$.
\begin{enumerate}
 \item If $T$ is ideal irreducible, then $\ker(r(T)-T)=\mathrm{Span}\{x_0\}$;
\item If $T$ is  band irreducible and $\sigma$-order continuous, and $x_0^*$ is strictly positive, then $\ker(r(T)-T)=\mathrm{Span}\{x_0\}$.
\end{enumerate}
\end{theorem}
One can replace the spectral radius with a positive eigenvalue as follows.
\begin{lemma}\label{eig2}Suppose $Tx_0=\lambda x_0$ and $T^*x_0^*=\lambda x_0^*$ for some $x_0>0$ and $x_0^*>0$ and $\lambda\in\mathbb{R}$.
\begin{enumerate}
 \item\label{eig2i1} If $T$ is ideal irreducible, then $\ker(\lambda-T)=\mathrm{Span}\{x_0\}$;
\item\label{eig2i2} If $T$ is  band irreducible and $\sigma$-order continuous, and $x_0^*$ is strictly positive, then $\ker(\lambda-T)=\mathrm{Span}\{x_0\}$.
\item\label{eig2i3} If $T^*$ is band irreducible, then $\ker(\lambda-T)=\mathrm{Span}\{x_0\}$ and $\ker(\lambda-T^*)=\mathrm{Span}\{x_0^*\}$.
\end{enumerate}
\end{lemma}
\begin{proof}The proof of \eqref{eig2i1} and \eqref{eig2i2} is analogous to that of Theorem~5.2, p.~329, \cite{scha3}. It remains to prove \eqref{eig2i3}. 
Note that $T$ is ideal irreducible; see Exercise 16, p.~356, \cite{abr}. Applying \eqref{eig2i1} to $T$, we know $\ker(\lambda-T)=\mathrm{Span}\{x_0\}$. By Lemma~\ref{faci1}\eqref{faci1i1}, $x_0$ is quasi-interior. Thus it acts 
as a strictly positive functional on $X^*$ such that $T^{**}x_0=\lambda x_0^*$. Note that, being the adjoint of a positive operator, $T^*$ is order continuous. Hence applying \eqref{eig2i2} to $T^*$, 
we have $\ker(\lambda-T^*)=\mathrm{Span}\{x_0^*\}$.
\end{proof}
For \eqref{eig2i3}, we would like to remark that if $X$ is order continuous, then $T^*$ is band irreducible if and only if $T$ is ideal irreducible; see Exercise 16, p.~356, \cite{abr}.\par
Recall that if $\lambda$ is a simple pole of $R(\cdot,T)$ then $PX=\ker(\lambda-T)$ where $P$ is the spectral projection of $T$ for $\lambda$; see Corollary 6.40, p. 268, \cite{abr}.

\begin{lemma}\label{eig23}Suppose $r(T)$ is a pole of $R(\cdot,T)$ and $T$ satisfies one of the following:
\begin{enumerate}
 \item\label{eig23i1} $T$ is ideal irreducible;
\item\label{eig23i2} $T$ is band irreducible and $\sigma$-order continuous, and $T^*x_0^*=r(T)x_0^*$ for some strictly positive functional $x_0^*$\footnote{By Lemma~\ref{seig}, 
what we really require here is strict positivity of $x_0^*$, not its existence.}.
\end{enumerate}
Then $r(T)>0$ and it is a simple pole of $R(\cdot,T)$, $PX=\ker(r(T)-T)=\mathrm{Span}\{x_0\}$ for some $x_0>0$, and $P^*X^*=\ker(r(T)-T^*)=\mathrm{Span}\{x_0^*\}$,
where $P$ is the spectral projection of $T$ for $r(T)$.
\end{lemma}
For Case \eqref{eig23i1}, a proof  of the assertions except the last one can be found in \cite{scha2} (see Theorem~2 there), a complete proof can be found in \cite{sawa} (see Theorems~1 and 2 there). 
Variants of Case \eqref{eig23i2} can be found in \cite{grob1} and \cite{kit}. We include here a simple proof for the convenience of the reader.
\begin{proof}We prove \eqref{eig23i2} first. By Lemma~\ref{seig}, there exists $x_0>0$ such that $Tx_0=r(T)x_0$. By Lemma~\ref{facb2}\,\eqref{facb2i1}, $r(T)>0$.\par
Let $r(T)$ be a pole of order $m$. Denote by $A_{-n}$ the coefficient of $(\lambda-r(T))^{-n}$ in the Laurent expansion of 
$R(\lambda,T)$ at $r(T)$. Then $A_{-m}=(T-r(T))^{m-1}P$, and $P^*$ is the spectral projection of $T^*$ for $r(T)$. Since $T^*x_0^*=r(T)x_0^*$, 
we know $x_0^*\in P^*X^*$, $P^*x_0^*=x_0^*$. By Lemma~\ref{seig}, we can take $x>0$ such that $A_{-m}x>0$. Then $0<x_0^*(A_{-m}x)=(T^*-r(T))^{m-1}(P^*x_0^*)(x)=[(T^*-r(T))^{m-1}x_0^*](x)$. 
If $m\geq 2$, then $(T^*-r(T))^{m-1}x_0^*=0$, 
yielding a contradiction! Hence $m=1$, that is, $r(T)$ is a simple pole.\par 
By the remark preceding this lemma and Lemma~\ref{eig2}\,\eqref{eig2i2}, $PX=\ker(r(T)-T)=\mathrm{Span}\{x_0\}$. Thus 
$\mathrm{rank}(P^*)=\mathrm{rank}(P)=1$. It follows from $0\neq\ker(r(T)-T^*)\subset P^*X^*$ that $P^*X^*=\ker(r(T)-T^*)=\mathrm{Span}\{x_0^*\}$.
The proof of \eqref{eig23i1} is similar, as in view of Lemmas~\ref{seig} and \ref{faci1}, there exist $x_0>0$ such that $Tx_0=r(T)x_0$ and a strictly positive functional $x_0^*$ such that $T^*x_0^*=r(T)x_0^*$.
\end{proof}

\begin{lemma}\label{copa}
\begin{enumerate}
\item\label{copai1}Suppose $T^*x_0^*=\lambda x_0^*$ for some strictly positive functional $x_0^*$. Then for any $x\in X$ such that $Tx\geq \lambda x$ or $Tx\leq \lambda x$, we have $Tx_\pm=\lambda x_\pm$;
\item\label{copai2}Suppose $Tx_0=\lambda x_0$ for some quasi-interior point $x_0>0$. Then for any $x^*\in X^*$ such that $T^*x^*\geq \lambda x^*$ or $T^*x^*\leq \lambda x^*$, we have $T^*x^*_\pm=\lambda x^*_\pm$;
\item\label{copai3}Let $T$ be $\sigma$-order continuous. Suppose $Tx_0=\lambda x_0$ for some weak unit $x_0>0$. Then for any $\sigma$-order continuous $x^*\in X^*$ such that $T^*x^*\geq \lambda x^*$ or $T^*x^*\leq \lambda x^*$, 
we have $T^*x^*_\pm=\lambda x^*_\pm$.
\end{enumerate}
\end{lemma}
\begin{proof}\eqref{copai1} Note that $0=(T^*x_0^*-\lambda x_0^*)(x)=x_0^*(Tx-\lambda x)$.
Since $x_0^*$ is strictly positive, $Tx=\lambda x$. This in turn implies $\lambda |x|\leq T|x|$. Using what we have just proved, we have $T|x|=\lambda |x|$.
Hence, $Tx_\pm=\lambda x_\pm$.\par
\eqref{copai3} Note that $(T^*x^*-\lambda x^*)(x_0)=x^*(Tx_0)-\lambda x^*(x_0)=0$. Since $T^*x^*-\lambda x^*$ is $\sigma$-order continuous and $x_0$ is a weak unit, $T^*x^*-\lambda x^*=0$, i.e., $T^*x^*=\lambda x^*$. 
This in turn implies $\lambda |x^*|\leq T|x^*|$.
Since $|x^*|$ is also $\sigma$-order continuous, applying what we have just proved, we have $T^*|x^*|=\lambda |x^*|$. Hence $Tx^*_\pm=\lambda x^*_\pm$.\par
\eqref{copai2} can be proved either similarly as \eqref{copai3}, or via \eqref{copai1} since $x_0$ acts as a strictly positive functional on $X^*$ such that $T^{**}x_0=\lambda x_0$.\end{proof}

Recall that an operator $S\in L(X)$ is called \textbf{order weakly compact} if $S[0,x]$ is relatively weakly compact for all $x>0$. 
This is a large class of operators containing all compact operators, AM-compact operators and weakly compact operators.
Recall also that an operator is called strictly singular if it fails to be invertible on any infinite-dimensional closed subspaces. Strictly singular operators are order weakly compact; see Corollary~3.4.5, p.~193, \cite{mn}.
We say that $S$ is power compact (power weakly compact, etc) if $S^k$ satisfies the property for some $k\geq 1$.\par
The following lemma  handles $\sigma$-order continuity of eigenfunctionals. The idea of the proof has appeared in \cite{grob1,grob2, sab}, ect.

\begin{lemma}\label{fac13}
\begin{enumerate}
 \item\label{fac13i1} If $T$ is $\sigma$-order continuous and order weakly compact, then $T$ is $\sigma$-order-to-norm continuous.
\item\label{fac13i2} If $T$ is power $\sigma$-order-to-norm continuous and $T^*x^*=\lambda x^*$ for some $\lambda\neq0$, then $x^*$ is $\sigma$-order continuous.
\item\label{fac13i3} If $T$ is $\sigma$-order continuous and power order weakly compact and $T^*x^*=\lambda x^*$ for some $\lambda\neq 0$, then $x^*$ is $\sigma$-order continuous.
\end{enumerate}
\end{lemma}
\begin{proof}\eqref{fac13i1} Take $x_n\downarrow0$. Then $Tx_n\downarrow0$. By Eberlein-Smulian theorem, $Tx_{n_j}\rightarrow y$ weakly for some $(n_j)$ and $y\in X$. Since $Tx_{n_j}$ is decreasing, 
it is straightforward verifications that $y=\inf_jTx_{n_j}=\inf_nTx_n=0$. Hence $||Tx_{n_j}||\rightarrow0$ by Dini theorem (see Theorem~3.52, p.~174, \cite{ali}). This in turn implies $||Tx_n||\rightarrow 0$.\par
For \eqref{fac13i2}, simply note that $x^*\in \mathrm{Range}(T^{k*})$. \eqref{fac13i3} follows from \eqref{fac13i1} and \eqref{fac13i2}.
\end{proof}

We end this section with the following variant of Theorem~\ref{eig02}, Lemmas~\ref{eig2}\,\eqref{eig2i3} and \ref{eig23}.
\begin{proposition}\label{eig233}Suppose $T$ is power order weakly compact such that $Tx_0=\lambda x_0$ and $T^*x_0^*=\lambda x_0^*$ for some $x_0>0$ and $x_0^*>0$. 
If $T$ satisfies one of the following:
\begin{enumerate}
 \item\label{eig233i1} $T$ is ideal irreducible;
\item\label{eig233i2} $T$ s band irreducible and $\sigma$-order continuous,
\end{enumerate}
then $\ker(\lambda-T)=\mathrm{Span}\{x_0\}$ and $\ker(\lambda-T^*)=\mathrm{Span}\{x_0^*\}$.
\end{proposition}
\begin{proof}We only prove \eqref{eig233i2}; the proof of \eqref{eig233i1} is similar. By Lemma~\ref{facb2}\,\eqref{facb2i1}, $x_0$ is a weak unit and $\lambda>0$. 
Hence $x_0^*$ is $\sigma$-order continuous by Lemma~\ref{fac13}\,\eqref{fac13i3}, and is strictly positive by Lemma~\ref{facb2}\,\eqref{facb2i2}. Therefore, by Lemma~\ref{eig2}\,\eqref{eig2i2}, 
$\ker(\lambda-T)=\mathrm{Span}\{x_0\}$.\par
Without loss of generality, assume $\lambda=1$. Suppose that $\dim\ker(1-T^*)>1$. Then there exists $x^*\in \ker(1-T^*)$ with $x^*_+>0$ and $x^*_->0$. 
By Lemma~\ref{fac13}\,\eqref{fac13i3}, $x^*$ is $\sigma$-order continuous. Since $x_0$ is a weak unit, Lemma~\ref{copa}\,\eqref{copai3}
implies that $x^*_\pm\in\ker(1-T^*)$. Let $k\geq 1$ be such that $T^k$ is order weakly compact. Then $(T^k)^*x^*_\pm=x^*_\pm$. It follows that
\begin{align*}0=&x^*_+\wedge x^*_-(x_0)=\inf_{0\leq x\leq x_0}(x^*_+(x_0-x)+x^*_-(x))\\
=&\inf_{0\leq x\leq x_0}((T^k)^*x^*_+(x_0-x)+(T^k)^*x^*_-(x))=\inf_{0\leq x\leq x_0}(x^*_+(x_0-T^kx)+x^*_-(T^kx)).\end{align*}
So we can take $x_n\in[0,x_0]$ such that $x^*_+(x_0-T^kx_n)+x^*_-(T^kx_n)\rightarrow 0$. Using order weak compactness of $T^k$ and Eberlein-Smulian theorem, we can assume, by passing to a subsequence, 
that $T^kx_n\rightarrow y\in[0,T^kx_0]=[0,x_0]$ weakly. In particular, we have $x^*_+(x_0-y)+x^*_-(y)=0$. It follows that $x^*_+(x_0-y)=x^*_-(y)=0$. 
Since $x^*_\pm$ are both $\sigma$-order continuous and lie in $\ker(1-T^*)$, they are both strictly positive by Lemma~\ref{facb2}\,\eqref{facb2i2}. This forces $x_0-y=y=0$. Thus $x_0=0$, 
which is absurd. Therefore, $\dim\ker(\lambda-T^*)=1$, and $\ker(\lambda-T^*)=\mathrm{Span}\{x_0^*\}$.
\end{proof}

\section{Main Theorems}\label{sec5}
The following is straightforward.
\begin{lemma}\label{copa0}
Suppose $0\leq S\leq T$. Then $T=S$ if any of the following is satisfied:
\begin{enumerate}
\item\label{copa0i1} $T^*x_0^*=S^*x_0^*$ for some strictly positive functional $x_0^*$;
\item\label{copa0i2} $Tx_0=S x_0$ for some quasi-interior point $x_0>0$;
\item\label{copa0i3} $T$ is $\sigma$-order continuous and $Tx_0=S x_0$ for some weak unit $x_0>0$.
\end{enumerate}
\end{lemma}

\begin{lemma}\label{gc0}Suppose $0\leq S\leq T$ and $T$ is ideal irreducible. Then $T=S$ if any of the following is satisfied:
\begin{enumerate}
\item\label{gc0i1} $T^*x_0^*=\lambda x_0^*$ for some $x_0^*>0$ and $ Sx_0\geq \lambda x_0$ for some $x_0>0$;
\item\label{gc0i2} $Tx_0=\lambda x_0$ for some $x_0>0$ and $S^*x_0^*\geq \lambda x_0^*$ for some $x_0^*>0$.
\end{enumerate}
Suppose $0\leq S\leq T$ and $T$ is band irreducible and $\sigma$-order continuous. Then $T=S$ if any of the following is satisfied:
\begin{enumerate}
 \item[(1')]\label{gc0i11} $T^*x_0^*=\lambda x_0^*$ for some strictly positive $x_0^*>0$ and $Sx_0\geq \lambda x_0$ for some $x_0>0$;
\item[(2')]\label{gc0i21} $Tx_0=\lambda x_0$ for some $x_0>0$ and $S^*x_0^*\geq \lambda x_0^*$ for some $\sigma$-order continuous $x_0^*>0 $.
\end{enumerate}
\end{lemma}
\begin{proof}
We only prove (2'); the other cases can be proved similarly. By Lemma~\ref{facb2}\,\eqref{facb2i1}, we know $x_0$ is a weak unit. Now note $\lambda x_0^*\leq S^*x_0^*\leq T^*x_0^*$. 
Hence, $\lambda x_0^*=S^*x_0^*=T^*x_0^*$ by Lemma~\ref{copa}\,\eqref{copai3}. By Lemma~\ref{facb2}\,\eqref{facb2i2}, $x_0^*$ is strictly positive. Therefore, it follows from $T^*x_0^*=S^*x_0^*$ that $T=S$ 
by Lemma~\ref{copa0}\,\eqref{copa0i1}.
\end{proof}
\begin{lemma}\label{omco}Suppose $T$ and $S$ are compact, $0\leq S\leq T$ and $r(T)=r(S)$. Then $T=S$ if $T$ is either ideal irreducible, or band irreducible and $\sigma$-order continuous.
\end{lemma}
\begin{proof}Suppose that $T$ is band irreducible and $\sigma$-order continuous. By Schaefer-Grobler Theorem~\cite{grob1} (see also Corollary~9.33, p.~367, \cite{abr}), $r(S)=r(T)>0$. Since $T$ and $S$ are both compact,
by Krein-Rutman Theorem, we can take $x_0>0$ and $x_0^*>0$ such that $Tx_0=r(T)x_0$ and $S^*x_0^*=r(S)x_0^*$. Since $S$ is also $\sigma$-order continuous, $x_0^*$ is $\sigma$-order continuous by Lemma~\ref{fac13}\,\eqref{fac13i3}.
This completes the proof by Lemma~\ref{gc0}\,(2').\par
The case when $T$ is ideal irreducible can be proved by using similar arguments and de Pagter's theorem~\cite{pag} that compact ideal irreducible operators are non-quasinilpotent.
\end{proof}

\begin{theorem}\label{mainc1}Suppose $0\leq S\leq T$, $r(T)=r(S)$, and $S^k$ is non-zero and compact for some $k\geq 1$. Then $T=S$ if $T$ is either ideal irreducible or band irreducible and $\sigma$-order continuous.
\end{theorem}
\begin{proof}
We only prove the band irreducible case; the other case can be proved similarly. Without loss of generality, assume $||T||<1$. Put $\widetilde{T}=\sum_1^\infty T^m$ and $\widetilde{S}=\sum_1^\infty S^m$.
Recall that $r(T)\in \sigma(T)$ and $r(S)\in\sigma(S)$. Thus by the spectral mapping theorem, it is easily seen that $$r(\widetilde{T}T^k\widetilde{T})=r(T)^k\left(\sum_1^\infty r(T)^m\right)^2= 
r(S)^k\left(\sum_1^\infty r(S)^m\right)^2=r(\widetilde{S}S^k\widetilde{S}).$$Together with $\widetilde{T}T^k\widetilde{T}\geq \widetilde{T}S^k\widetilde{T}\geq \widetilde{S}S^k\widetilde{S}$, 
this implies $r(\widetilde{T}S^k\widetilde{T})=r( \widetilde{S}S^k\widetilde{S})$.\par
Recall that $\widetilde{T}$ is $\sigma$-order continuous. Hence, so is $ \widetilde{T}S^k\widetilde{T}$. By Lemma~\ref{chb2}\,\eqref{chb2i2}, $\widetilde{T}$ is expanding, 
hence so is $ \widetilde{T}S^k\widetilde{T}$; in particular, $ \widetilde{T}S^k\widetilde{T}$ is band irreducible. Finally note that since $S^k$ is compact, so are $\widetilde{T}S^k\widetilde{T}$ and $\widetilde{S}S^k\widetilde{S}$. 
Applying Lemma~\ref{omco} to $0\leq\widetilde{S}S^k\widetilde{S}\leq\widetilde{T}S^k\widetilde{T}$, we have $\widetilde{T}S^k\widetilde{T}=\widetilde{S}S^k\widetilde{S}$.\par
Note that $\widetilde{T}S^k\widetilde{T}\geq \widetilde{T}S^k\widetilde{S}\geq \widetilde{S}S^k\widetilde{S}$. Hence, $\widetilde{T}S^k\widetilde{T}= \widetilde{T}S^k\widetilde{S}$. 
Since $\widetilde{T} $ is strictly positive and $S^k\widetilde{T}\geq S^k\widetilde{S}$, we have $S^k\widetilde{T}= S^k\widetilde{S}$. If $S^kx=0$ for some $x>0$, then $S^k\widetilde{T}x=\widetilde{S}S^kx=0$. 
But $\widetilde{T}x$ is a weak unit, forcing $S^k=0$, which is absurd. Hence $S^k$ is strictly positive. Now it follows from $S^k\widetilde{T}=S^k\widetilde{S}$ and $\widetilde{T}\geq \widetilde{S}$ that 
$\widetilde{T}=\widetilde{S}$. This in turn implies $T=S$.
\end{proof}

We are now ready to present a generalization of Theorem~\ref{crad0} to operators on arbitrary Banach lattices.
\begin{corollary}\label{ckm11}Suppose $0\leq S\leq T$ and $r(T)=r(S)$. Then $T=S$ if one of the following is satisfied:\par
\begin{enumerate}
 \item\label{ckm11i1} $T$ is power compact, and is either ideal irreducible or band irreducible and $\sigma$-order continuous;
\item\label{ckm11i2}  $S$ is power compact, and either $S$ is ideal irreducible, or $S$ is band irreducible and $T$ is $\sigma$-order continuous.
\end{enumerate}
\end{corollary}
\begin{proof}\eqref{ckm11i1} By Lemma~\ref{faci1}\,\eqref{faci1i3} or Lemma~\ref{facb2}\,\eqref{facb2i3}, each power of $T$ is non-zero, hence $r(S)=r(T)>0$ by Corollary~4.2.6, p.~267, \cite{mn}. 
In particular, this implies that each power of $S$ is non-zero. By Aliprantis-Burkinshaw's Cube theorem (Theorem~5.14, p.~283, \cite{ali}), we know $S$ is also power compact. The desired result now follows from Theorem~\ref{mainc1} immediately.\par
\eqref{ckm11i2} Assume $S^k$ is compact. By Lemma~\ref{faci1}\,\eqref{faci1i3} or Lemma~\ref{facb2}\,\eqref{facb2i3}, $S^k>0$. Note that since $S$ is irreducible, so is $T$. Now apply Theorem~\ref{mainc1} again.
\end{proof}
\begin{remark}Lemma~\ref{omco}, Theorem~\ref{mainc1} and Corollary~\ref{ckm11} still hold if we replace compactness involved by strict singularity and assume $r(T)>0$. 
The same lines of arguments with minor modifications will work. For example, let's look at the band irreducible case of Lemma~\ref{omco}. Suppose $S$ and $T$ are now strictly singular. 
Since $0<r(T)\in \sigma(T)$ and $0<r(S)\in\sigma(S)$, $r(T) $ and $r(S)$ are poles of $R(\cdot,T)$ and $R(\cdot,S)$, respectively; 
see Exercise~8, p.~314 and Corollary~7.49, p.~303, \cite{abr}. Replacing Krein-Rutman Theorem with Lemma~\ref{seig}, we get $x_0$ and $x_0^*$ as before. Since $S$ is order weakly compact, Lemma~\ref{fac13} again implies that $x_0^*$ is 
$\sigma$-order continuous. Thus
Lemma~\ref{omco} holds. Theorem~\ref{mainc1} holds because the set of strictly singular operators 
also forms an ideal of $L(X)$; see Corollary~4.62, p.~175, \cite{abr}. Corollary~\ref{ckm11} holds because the power property also holds for strictly singular operators 
(that is, if $0\leq S\leq T$ and $T$ is strictly singular, then $S^4$ is strictly singular; see Corollary~4.2, \cite{ped}).
\end{remark}

Motivated by an idea from \cite{crad}, we can also prove a variant of Corollary~\ref{ckm11} replacing power compactness condition with the spectral radius being a pole of the resolvent. We need the following two lemmas.
\begin{lemma}[\cite{kar}]\label{kkk}If $r(T)=1$ is a simple pole of $R(\cdot,T)$ then $\frac{1}{n}\sum_{i=1}^nT^i\rightarrow P$, the spectral projection of $T$ for $r(T)=1$.
\end{lemma}
\begin{lemma}\label{csrx}Suppose $T_n\rightarrow T$ and $\sigma(T)\backslash\sigma_{per}(T)$ is closed. Then $r(T_n)\rightarrow r(T)$.
\end{lemma}
\begin{proof}It is well known that $\limsup_nr(T_n)\leq r(T)$; see \cite{new}.\par We proceed to prove $r(T)\leq \liminf_nr(T_n)$. The argument is an imitation of the proof of continuity of spectral radius on compact operators. Assume $r(T)>\liminf_nr(T_n)$. Take $\epsilon>0$ such that $r(T)>\liminf_nr(T_n)+2\varepsilon$. By passing to a subsequence, we may assume $r(T)>r(T_n)+\epsilon$ for all $n\geq 1$.\par Since $\sigma(T)\backslash\sigma_{per}(T)$ is closed, we can take $\delta>0$ small enough so that $2\delta<\varepsilon$ and $\sigma(T)\backslash\sigma_{per}(T)\subset \{z:|z|<r(T)-2\delta\}$. Now for $j=\pm1$, define curves $\gamma_j(t)=[r(T)+j\delta]e^{it}$, $0\leq t\leq 2\pi$. Then by Cauchy integral theorem, we have
$$\int_{\gamma_1-\gamma_{-1}} R(\lambda,T_n)\mathrm{d}\lambda=0,\;\forall\;n\geq 1.$$
But on the other hand, note that in any unital Banach algebra, for $a$ invertible and $||x||$ small enough,
$||(a-x)^{-1}-a^{-1}||\leq\frac{||x||\,||a^{-1}||^2}{1-||x||\,||a^{-1}||}$; see p.~5, \cite{mul}. Using this, it is easy to see that
$R(\cdot,T_n)\rightarrow R(\cdot,T)$ uniformly on $\gamma_j$'s. Therefore
$$P=\frac{1}{2\pi i}\int_{\gamma_1-\gamma_{-1}} R(\lambda,T)\mathrm{d}\lambda=\lim_n\frac{1}{2\pi i}\int_{\gamma_1-\gamma_{-1}} R(\lambda,T_n)\mathrm{d}\lambda=0,$$
where $P$ is the spectral projection of $T$ for $\sigma_{per}(T)$. This is absurd.
\end{proof}

\begin{theorem}\label{coo}Let $0\leq S\leq T$ be such that $r(T)=r(S)$.
\begin{enumerate}
 \item\label{cooi1} Suppose that $r(T)$ is a pole of $R(\cdot,T)$, then $T=S$ if $T$ is either ideal irreducible, or band irreducible and $\sigma$-order continuous with $\sigma$-order continuous\footnote{By Lemma~\ref{seig}, what we require here is 
the $\sigma$-order continuity of $x_0^*$, not its existence.} $x_0^*>0$ such that $T^*x_0^*=r(T)x_0$;
\item\label{cooi2} Suppose that $r(S)$ is a pole of $R(\cdot,S)$, then $T=S$ if either $S$ is ideal irreducible, or $S$ is band irreducible with $\sigma$-order continuous $x_0^*>0$ such that $S^*x_0^*=r(S)x_0$ and $T$ is $\sigma$-order continuous.
\end{enumerate}
\end{theorem}

\begin{proof}\eqref{cooi1} Suppose first $T $ is band irreducible and $\sigma$-order continuous, and $T^*x_0^*=r(T)x_0^*$ for some $\sigma$-order continuous $x_0^*>0$. 
Let $P$ be the spectral projection of $T$ for $r(T)$. We know that $x_0^*$ is strictly positive by Lemma~\ref{facb2}\,\eqref{facb2i2}. Hence, by Lemma~\ref{eig23}, $r(T)>0$ is a simple pole of $R(\cdot,T)$, 
$PX=\ker(r(T)-T)=\mathrm{Span}\{x_0\}$ for a weak unit $x_0>0$ and $P^*X^*=\ker(r(T)-T^*)=\mathrm{Span}\{x_0^*\}$. Without loss of generality, we can assume $r(T)=1$ and 
$P(\cdot)=x_0^*(\cdot)x_0$. Then $PT=P$ is compact, expanding and $\sigma$-order continuous.\par
Since $r(T)$ is a simple pole of $R(\cdot,T)$ and $PS$ and $PT$ are compact, we have, by Lemmas~\ref{kkk} and \ref{csrx}, \begin{align*}1=&\lim_nr\left(\frac{1}{n}\sum_1^nS^iS\right)
\leq\lim_nr\left(\frac{1}{n}\sum_1^nT^iS\right)=r(PS)\\
\leq &r(PT)=\lim_nr\left(\frac{1}{n}\sum_1^nT^iT\right)=1.\end{align*}It follows that $r(PT)=r(PS)=1$. Now applying Lemma~\ref{omco} to $0\leq PS \leq PT$, we know $PT=PS$. Hence, $T=S$ due to strict positivity of $P$.
The case when $T$ is ideal irreducible can be proved similarly.\par
For \eqref{cooi2}, let $P$ be the spectral projection of $S$ for $r(S)$. Then as before, we can show that $PS$ is (strongly) expanding, and hence so is $PT$. A similar argument also gives $r(PT)=r(PS)$.
Now applying Lemma~\ref{omco} to $0\leq PS \leq PT$ again, we obtain $PT=PS$, and thus $T=S$ due to strict positivity of $P$.\end{proof}

\begin{remark}Recall that power compact operators are essentially quasinilpotent; see Definition~7.46, p.~302, \cite{abr}. Recall also that power compact irreducible operators are non-quasinilpotent 
(Corollary~4.2.6, p.~267, \cite{mn}). Hence, the spectral radius of a power compact irreducible operator is a pole of the resolvent; see Corollary~7.49, p.~303, \cite{abr}. Therefore, it is easily seen that Corollary~\ref{ckm11} 
and Theorem~\ref{crad0} can be deduced from Theorem~\ref{coo}.\end{remark}
\section{Further Remarks}\label{lastsss}
When the dominating operator is irreducible, the comparison theorem still holds if both operators bear some even weaker spectral conditions. We begin with a slight generalization of Lemma~\ref{seig}.
\begin{lemma}\label{gc1}
Let $T$ be such that $r(T)^k$ is a pole of $R(\cdot,T^k)$ for some $k\geq 2$. If $r(T)>0$ or $T$ is strictly positive, then $T$ has a positive eigenvector for $r(T)$.
\end{lemma}
\begin{proof}
Suppose that $r(T)^k$ is a pole of order $m\geq 1$. Then Lemma~\ref{seig} implies $(T^k-r(T)^k)^{m-1}P=A_{-m}>0$
where $P$ is the spectral projection of $T^k$ for $r(T)^k$ and $A_{-m}$ is the coefficient of $(\lambda-r(T)^k)^{-m}$ in the Laurent expansion of $R(\lambda,T^k)$ at $r(T)^k$.\par
Take $v>0$ such that $A_{-m}v>0$. Then by the given conditions, $u=:(r(T)^{k-1}+r(T)^{k-2}T+\dots+T^{k-1})A_{-m}v>0$.
Now $(T-r(T))u=(T^k-r(T)^k)A_{-m}v=A_{-m-1}v=0$. Hence $u$ is a positive vector as required.
\end{proof}
\begin{proposition}\label{coo1}Suppose $0\leq S\leq T$, $r(T)=r(S)$, and there exist $k,m\geq 1$ such that $r(T)^k$ is a pole of $R(\cdot,T^k) $ and $r(S)^m$ is a pole of $R(\cdot,S^m)$. 
Then $T=S$ if one of the following is satisfied:
\begin{enumerate}
 \item\label{coo1i1} $T $ is ideal irreducible;
\item\label{coo1i2} $T$ is band irreducible and $\sigma$-order continuous and $S$ is power $\sigma$-order-to-norm continuous.
\end{enumerate}
\end{proposition}
\begin{proof}
We only prove \eqref{coo1i2}; the proof of \eqref{coo1i1} is similar. By Lemma~\ref{facb2}\,\eqref{facb2i3}, $T$ is strictly positive. Hence, by Lemma~\ref{seig} or \ref{gc1}, $Tx_0=r(T)x_0$ for some $x_0>0$.
It follows from Lemma~\ref{facb2}\,\eqref{facb2i1} that $r(S)=r(T)>0$. Now
Lemma~\ref{seig} or \ref{gc1} applied to $S^*$ yields $x_0^*>0$ such that $S^*x_0^*=r(S)x_0^*$. We know $x_0^*$ is $\sigma$-order continuous by Lemma~\ref{fac13}. Hence, $T=S$ by Lemma~\ref{gc0}\,(2').
\end{proof}
It is clear that Corollary~\ref{ckm11}\,\eqref{ckm11i1} can be deduced from this proposition.\par
When the dominated operator is irreducible and bears similar spectral conditions, we are able to establish comparison theorems for commuting operators.
\begin{lemma}\label{ss00}Suppose $0\leq S\leq T$, $r(T)=r(S)$, $Sx_0=r(S)x_0$ and $S^*x_0^*=r(S)x_0^*$ for some $x_0>0$ and $x_0^*>0$. Suppose also $TS=ST$. Then $T=S$ if one of the following is satisfied:
\begin{enumerate}
 \item\label{ss00i1} $S$ is ideal irreducible,
\item\label{ss00i2} $S$ is band irreducible, $T$ is $\sigma$-order continuous and $x_0^*$ is strictly positive.
\end{enumerate}
\end{lemma}
\begin{proof}We only prove \eqref{ss00i2}; the proof of \eqref{ss00i1} is similar. Note that $S$ is band irreducible and $\sigma$-order continuous. Hence, $\ker(r(S)-S)=\mathrm{Span}\{x_0\}$ by Lemma~\ref{eig2}\,\eqref{eig2i2}.
Since $S(Tx_0)= TSx_0=r(S)(Tx_0)$, $Tx_0\in\ker(r(S)-S)$. Hence $Tx_0=cx_0$ for some $c\in\mathbb{R}$. Clearly, $0\leq c\leq r(T)$. 
On the other hand, $cx_0=Tx_0\geq Sx_0=r(T)x_0$ implies $c\geq r(T)$. Hence, $c=r(T)$. It follows that $Tx_0=r(T)x_0=Sx_0$. Since $x_0$ is a weak unit, $T=S$ by Lemma~\ref{copa0}\,\eqref{copa0i3}.
\end{proof}
\begin{proposition}\label{rc2}Let $0\leq S\leq T$ be such that $r(T)=r(S)$ and $ST= TS$. Suppose there exists $m\geq 1$ such that $r(S)^m$ is a pole of $R(\cdot,S^m)$. 
Then $T=S$ if one of the following is satisfied:
\begin{enumerate}
 \item\label{rc2i1} $S$ is ideal irreducible,
\item\label{rc2i2} $S$ is band irreducible and power $\sigma$-order-to-norm continuous, and $T$ is $\sigma$-order continuous.
\end{enumerate}
\end{proposition}
\begin{proof}We only prove \eqref{rc2i2}; the proof of \eqref{rc2i1} is similar. Clearly, $S$ is band irreducible and $\sigma$-order continuous. By Lemma~\ref{facb2}\,\eqref{facb2i3}, it is strictly positive. 
Applying Lemma~\ref{seig} or \ref{gc1}, we have $Sx_0=r(S)x_0$ for some $x_0>0$. This in turn implies $r(S)>0$ by Lemma~\ref{facb2}\,\eqref{facb2i1}. 
Now applying Lemma~\ref{seig} or \ref{gc1} to $S^*$, we have $S^*x_0^*=r(S)x_0^*$ for some $x_0^*>0$. By Lemma~\ref{fac13}, $x_0^*$ is $\sigma$-order continuous, and thus is strictly positive by Lemma~\ref{facb2}\,\eqref{facb2i2}. 
This completes the proof by Lemma~\ref{ss00}\,\eqref{ss00i2}.
\end{proof}
It deserves mentioning that the commutation condition in Lemma~\ref{ss00} and Proposition~\ref{rc2} is equivalent to semi-commutation. The following is a slight generalization of Corollary~3.3 in \cite{drnx}.
\begin{lemma}\label{scc}Let $U,V\in L(X)$ be such that $UV\geq VU$ or $UV\leq VU$. Suppose $Ux_0=\lambda x_0$ and $U^*x_0^*=\lambda x_0^*$ for some $x_0>0$ and strictly positive $x_0^*$. Then $UV=VU$ if
one of the following is satisfied:
\begin{enumerate}
 \item\label{scci1} $x_0$ is a quasi-interior point,
\item\label{scci2} $x_0$ is a weak unit, and $U$ and $V$ are $\sigma$-order continuous.
\end{enumerate}
\end{lemma}
\begin{proof}We only prove \eqref{scci2}; the proof of \eqref{scci1} is similar.
Note that $x_0^*((UV-VU)x_0)=(U^*x_0^*)(Vx_0)-x_0^*(VUx_0)=\lambda x_0^*(Vx_0)-x_0^*(V\lambda x_0)=0$. Since $x_0^*$ is strictly positive, we have $(UV-VU)x_0=0$. 
Note also $UV-VU$ is $\sigma$-order continuous. Thus $x_0 $ being a weak unit yields $UV-VU=0$ by Lemma~\ref{copa0}\,\eqref{copa0i3}.
\end{proof}
\medskip
We would like to mention a recent preprint \cite{kit2}, where several similar comparison theorems were obtained independtently using different techniques.\\
\medskip

\noindent\textbf{Acknowledgements.} The author is grateful to his advisor Dr.~Vladimir G.~Troitsky for suggesting these problems and reading the manuscript. 
The author would also like to express his sincere gratitude to the referee for making
many beneficial suggestions and bringing some references into the author's attention.

\end{document}